\documentclass[a4paper]{amsart}
\usepackage[american]{babel}
\usepackage[T1]{fontenc}
\usepackage[utf8]{inputenc}
\usepackage{amsmath,amsfonts,amstext, dsfont}
\usepackage{amsthm,amssymb,amsmath,amscd,graphicx,epsfig,hhline,mathrsfs,latexsym,url}
\usepackage[neverdecrease]{paralist}
\usepackage{yfonts,dsfont}
\usepackage[numbers,sort&compress]{natbib}

\newcommand{\mc}{\mathcal}

\newcommand{\mb}{\mathbb}
\newcommand{\mr}{\mathrm}

\def\a{\beta} 
\def\M{\mc{M}} 

\newtheorem*{R}{Remark}
\newtheorem{T}{Theorem} \newtheorem{Le}{Lemma}
 \newtheorem{A}{Proposition}
\theoremstyle{definition}\newtheorem{D}{Definition}[section]
\theoremstyle{plain}

\newcommand{\lin}{\text{lin}}
\newcommand{\clin}{\overline{\lin}}

\title{On the Entangled Ergodic Theorem}
\author{Tanja Eisner and D\'avid Kunszenti-Kov\'acs}
\address{Institute of Mathematics\\University of T\"ubingen\\Auf der
  Morgenstelle 10\\72076 T\"ubingen\\Germany}
\email{talo@fa.uni-tuebingen.de, daku@fa.uni-tuebingen.de}


\begin{document}

\begin{abstract}
We study the convergence of the so-called entangled ergodic averages 
$$
\frac{1}{N^k}\sum_{n_1,\ldots,n_k=1}^{N}T_m^{n_{\alpha(m)}}A_{m-1}T_{m-1}^{n_{\alpha(m-1)}}A_{m-2}\ldots A_1T_1^{n_{\alpha(1)}},
$$
where $k\leq m$ and
$\alpha:\left\{1,\ldots,m\right\}\to\left\{1,\ldots,k\right\}$ is a
surjective map. We show that, on general Banach spaces and without any
restriction on the partition $\alpha$, the above averages converge
strongly as $N\to \infty$ under some quite weak compactness
assumptions on the operators $T_j$ and $A_j$. A formula for the limit
based on the spectral analysis of the operators $T_j$ 
and the continuous version of the result are presented as well. 
\end{abstract}

\maketitle

\section{Introduction}

The classical mean ergodic theorem has inspired many mathematicians and led to several generalisations and extensions. We mention Berend, Lin, Rosenblatt, Tempelman \cite{berend/lin/rosenblatt/tempelman} for modulated and subsequential ergodic theorems and e.g.~Kra \cite{kra:2007} for an overview on multilpe ergodic theorems as well as for the history of the subjects and further references.

In this note we study a further extension of the mean ergodic theorem, namely the so-called entangled ergodic theorem.  
Let $\alpha:\left\{1,\ldots,m\right\}\to\left\{1,\ldots,k\right\}$ be
a surjective map for some positive integers $k\leq m$, and take
$T_1,\ldots,T_m$ and $A_1,\ldots,A_{m-1}$ to be linear operators on a Banach
space $X$.
We investigate the convergence of the entangled Ces\`aro means
\begin{equation}\label{eq:entangled-cesaro}
\frac{1}{N^k}\sum_{n_1,\ldots,n_k=1}^{N}T_m^{n_{\alpha(m)}}A_{m-1}T_{m-1}^{n_{\alpha(m-1)}}A_{m-2}\ldots A_1T_1^{n_{\alpha(1)}}.
\end{equation}

This type of ergodic theorems was introduced by Accardi, Hashimoto, Obata \cite{accardi/hashimoto/obata:1998} motivated by quantum stochastics and was then studied by Liebscher \cite{liebscher:1999} and F. Fidaleo 
\cite{fidaleo:2007,fidaleo:2010}. 
In their studies, $T_1=\ldots=T_m=:U$ is
a unitary operator on a Hilbert space. Besides a technical assumption of Liebscher \cite{liebscher:1999},
there were two basic situations in which weak or strong convergence of
the entangled ergodic averages could be proved. The first is when $A_j$ are arbitrary and the unitary operator $U$ is almost periodic (see Definition \ref{def:alm-per} below),
see Liebscher \cite{liebscher:1999} and Fidaleo \cite{fidaleo:2010}. 
The second is when the operators $A_j$ are compact and $U$ is arbitrary unitary, see Fidaleo \cite{fidaleo:2007}.

In this paper we consider a more general situation. We assume the operators $T_j$
belong to a large class including power bounded operators on reflexive
Banach spaces. Further, we require a quite general compactness condition on the pairs $(A_j, T_j)$ generalising both of the above cases. More precisely, we make the
following assumptions. 
\begin{itemize}
\item[(A1)](Weakly compact orbits of $T_j$)

\noindent 
The operator $T_m$ is power bounded and totally ergodic and each of $T_1,\ldots,
T_{m-1}$ has 
relatively weakly compact orbits, i.e., $\{T_j^n x:\,
n\in \mb{N}\}$ is relatively compact in $X$ in the weak
topology for every $x\in X$ and 
$1\leq j\leq m-1$.

\item[(A2)](Joint compactness of $(A_j,T_j)$)

\noindent Every $A_j$ is compact on the orbits of $T_j$, i.e., 
$\{A_jT_j^n x:\, n\in \mb{N}\}$ is relatively compact in $X$
for every $x\in X$ and $1\leq j\leq m-1$.
\end{itemize}
Recall that an operator $T$ is called totally ergodic if 
the operators
$\lambda T$ are mean ergodic for every $\lambda\in\mb{T}$, $\mb{T}$
the unit circle. Every operator with relatively weakly compact orbits
is automatically totally ergodic. 
Note that assumption (A1) is not very restrictive. As mentioned above, every power
bounded operator on a reflexive Banach space has relatively weakly
compact orbits by the Banach--Alaoglu theorem. Another important
class of examples is given by power bounded positive operators on a Banach lattice
$L^1(\mu)$ preserving an order interval generated by a strictly positive function, see
e.g.~Schaefer \cite[Thm. II.5.10(f) and Prop. II.8.3]{schaefer:1974}
or \cite[Section I.1]{eisner-book} for further information.
While forming a large class, operators 
with relatively weakly compact orbits admit good asymptotic properties, see Theorem \ref{J-Gl-dL}
below. 

Under the assumptions (A1) and (A2) we show that the entangled Ces\`aro means
(\ref{eq:entangled-cesaro}) converge strongly and describe their limit
operator, see Theorem \ref{thm-general-case}. It turns out that only
specially interacting  
projections corresponding to unimodular
eigenvalues of $T_j$ (in combination with the operators $A_j$) contribute to the limit. 

The paper is organised as follows. We first treat the special case
assuming that all but the last $T_j$ are almost periodic
(Section \ref{section:almost-periodic}). Here we show how to reduce the
problem to the case when $T_1=\ldots=T_{m-1}$ and $A_1=\ldots=A_{m-1}$. In the context of Hilbert spaces and pair partitions, this case was considered recently in Fidaleo \cite{fidaleo:2010}. In
Section \ref{section:general-case} the general case is treated. We
further discuss the connection to non-commutative multiple ergodic theorems, them being 
an important recent field of research. 
In Section \ref{section:continuous-case}, we finally study the case of strongly continuous
semigroups and the strong convergence of the ergodic averages
\[
\frac{1}{t^k}\int_{[0,t]^k}T_m(s_{\alpha(m)})A_{m-1}T_{m-1}(s_{\alpha(m-1)})A_{m-2}\ldots A_1T_1(s_{\alpha(1)})\, ds_1\ldots ds_{k}.
\]
Note that the study of convergence of entangled ergodic averages for the continuous time scale seems not be new.

\vspace{0.3cm}

Before proceeding we first show by an example that one cannot drop assumption (A2)
even when $m=2$, the operator $T_1=T_2=:U$ is unitary and weakly
stable, i.e., satisfies  $\lim_{n\to\infty}U^n=0$ in the weak operator
topology. 

\begin{A}
There exists a Hilbert space $H$, a weakly stable unitary operator $U$ on $H$ and an operator $A\in\mc{L}(H)$ such that the Ces\`aro means
\[
\frac{1}{N}\sum_{n=1}^N U^nAU^n
\]
do not converge weakly.
\end{A}
\begin{proof}
Let $H:=\ell^2(\mb{Z})$ and consider the standard orthonormal base $\{e_b\}_{b\in\mb{Z}}$. Let then $U$ be the unitary left shift operator. Let further $(f(n))_{n\in\mb{N}^+}$ be a $0-1$ sequence that is not Ces\`aro-summable, and define $A$ by 
\[
Ae_b:=\left\{\begin{array}{lr} e_{f(-b)-b} & \mbox{ whenever } b<0,\\ e_b & \mbox{ whenever } b\geq0.\end{array}\right.
\]
Then $A$ is a bounded operator with $\|A\|=\sqrt3$, and we have $U^nAU^ne_0=e_{f(n)}$ for all $n\in\mb{N}$. Hence $\langle U^nAU^ne_0,e_0\rangle=1-f(n)$, and since $(f(n))_{n\in\mb{N}^+}$ is not Ces\`aro-summable, the Ces\`aro-means $\frac{1}{N}\sum_{n=1}^N U^nAU^n$ cannot converge weakly either.
\end{proof}

%
%
%
%

\section{The almost periodic case}\label{section:almost-periodic}

We first consider the situation when all operators $T_j$ are almost
periodic. In the case when $\alpha$ is a pair partition, $X$ is a
Hilbert space and the operators $T_j$ are all equal to a unitary
operator, this has been studied by Liebscher \cite{liebscher:1999} and Fidaleo \cite{fidaleo:2010}. 

\begin{D}\label{def:alm-per}
An operator $T\in\mc{L}(X)$ acting on a Banach space is called \emph{almost periodic} if it is power bounded (i.e.~$\sup_{n\in\mb{N}}\|T^n\|<\infty$) and satisifies 
$$
X=\overline{\mathrm{lin}}\left\{x\in X\big|\exists
  \lambda\in\mb{T}: Tx=\lambda x\right\}.
$$
\end{D}

The following result is the first step towards the general situation.

\begin{T}\label{AP}
Let $X$ be a Banach space, $T_1,\ldots,T_{m-1}$ be almost periodic
operators on $X$, $T_m\in\mc{L}(X)$  
be power bounded and totally ergodic, 
$A_1,\ldots,A_{m-1}\in\mc{L}(X)$, and
$\alpha:\left\{1,\ldots,m\right\}\to\left\{1,\ldots,k\right\}$ be
surjective for some $k\leq m$. 
Then the entangled Ces\`aro means
\begin{equation}\label{entEq}
\frac{1}{N^k}\sum_{n_1,\ldots,n_k=1}^{N}T_m^{n_{\alpha(m)}}A_{m-1}T_{m-1}^{n_{\alpha(m-1)}}A_{m-2}\ldots A_1T_1^{n_{\alpha(1)}}
\end{equation}
converge strongly as $N\to\infty$.
\end{T}

\begin{proof}
We shall proceed by induction on $m$, giving an explicit form for the limit. 

First recall that any operator with relatively weakly compact orbits
(in particular any almost periodic operator) is mean ergodic, see
e.g. \cite[Thm. I.2.9]{eisner-book}, and 
to have relatively weakly compact orbits 
is an invariant property under multiplication by a unimodular constant. 
We may therefore introduce the mean ergodic projections
\[
 P^{(j)}_{\lambda}:=\lim_{N\to\infty}\frac{1}{N}\sum_{n=1}^N (\lambda^{-1} T_j)^n,
\]
where $1\leq j\leq m$ and $|\lambda|=1$.

We now show that the limit of the entangled Ces\`aro means $(\ref{entEq})$ is given by the formal sum
\begin{equation}\label{limitEq}
 \sum_{\substack{\lambda_j\in\sigma_j\, (1\leq j\leq m)\\\prod_{i\in\alpha^{-1}(a)}\lambda_i=1\, (1\leq a\leq k)}}
P^{(m)}_{\lambda_m}A_{m-1}P^{(m-1)}_{\lambda_{m-1}}A_{m-2}\ldots A_1P^{(1)}_{\lambda_1}.
\end{equation}
Here, $\sigma_j$ denotes the point spectrum of $T_j$ ($1\leq j\leq m$)
and (\ref{limitEq}) should be understood as the strong limit of the net
\[
\left\{
 \sum_{\substack{\lambda_j\in F_j\, (1\leq j\leq m)\\\prod_{i\in\alpha^{-1}(a)}\lambda_i=1\, (1\leq a\leq k)}}
P^{(m)}_{\lambda_m}A_{m-1}P^{(m-1)}_{\lambda_{m-1}}A_{m-2}\ldots A_1P^{(1)}_{\lambda_1}
\right\}_{F_j\subset\sigma_j \text{finite} \,(1\leq j\leq m)}
.
\]

As seen above, this holds for $m=1$. Suppose we know that it holds for
any choice of $(m-2)$ almost periodic operators $T_j$ and $(m-2)$ bounded operators $A_i$ ($m\geq2$).
We may suppose without loss of generality that $\alpha(1)=1$.

If now $\alpha^{-1}(1)=\{1\}$, then (\ref{entEq}) can be written as
\[
 \left(\frac{1}{N^{(k-1)}}\sum_{n_2,\ldots,n_k=1}^{N}T_m^{n_{\alpha(m)}}A_{m-1}T_{m-1}^{n_{\alpha(m-1)}}A_{m-2}\ldots A_2T_2^{n_{\alpha(2)}}\right)A_1
\left(\frac{1}{N}\sum_{n_1=1}^NT_1^{n_1}\right).
\]
By the induction hypotheses and the joint continuity of multiplication in the strong operator topology, this expression converges to
\begin{eqnarray*}
 && \left(\sum_{\substack{\lambda_j\in\sigma_j\, (2\leq j\leq m)\\\prod_{i\in\alpha^{-1}(a)}\lambda_i=1\, (2\leq a\leq k)}}
P^{(m)}_{\lambda_m}A_{m-1}P^{(m-1)}_{\lambda_{m-1}}A_{m-2}\ldots A_2P^{(2)}_{\lambda_2}\right)A_1P^{(1)}_1\\
&=& \sum_{\substack{\lambda_j\in\sigma_j\, (1\leq j\leq m)\\\prod_{i\in\alpha^{-1}(a)}\lambda_i=1\, (1\leq a\leq k)}}
P^{(m)}_{\lambda_m}A_{m-1}P^{(m-1)}_{\lambda_{m-1}}A_{m-2}\ldots A_1P^{(1)}_{\lambda_1}.
\end{eqnarray*}

If, on the other hand, there exists an $1<l\leq m$ with $\alpha(l)=1$, then consider an eigenvector $x\in X$ of $T_1$ pertaining to some eigenvalue $\lambda\in\mb{T}$.
We can rewrite the entangled means applied to $x$ as

\begin{eqnarray*}
 &&\frac{1}{N^k}\sum_{n_1,\ldots,n_k=1}^{N}T_m^{n_{\alpha(m)}}A_{m-1}T_{m-1}^{n_{\alpha(m-1)}}A_{m-2}\ldots A_1T_1^{n_{\alpha(1)}}x\\
&=&\frac{1}{N^k}\sum_{n_1,\ldots,n_k=1}^{N}T_m^{n_{\alpha(m)}}A_{m-1}T_{m-1}^{n_{\alpha(m-1)}}A_{m-2}\ldots A_1\lambda^{n_{\alpha(1)}}x\\
&=&\frac{1}{N^k}\sum_{n_1,\ldots,n_k=1}^{N}T_m^{n_{\alpha(m)}}A_{m-1}T_{m-1}^{n_{\alpha(m-1)}}A_{m-2}\ldots (\lambda T_l)^{n_{\alpha(l)}} \ldots A_2T_2^{n_{\alpha(2)}}(A_1x).
\end{eqnarray*}

\noindent This reduces the problem to the $(m-2)$ almost periodic operators
$T_2,\ldots, \lambda T_l,\ldots, T_m$, and $(m-2)$ bounded operators
$A_2,\ldots,A_{m-1}$. The induction hypotheses together with
$x=P_{\lambda}^{(1)}x$ yields that the limit is

\begin{eqnarray*}
 &&\sum_{\substack{\lambda_j\in\sigma_j\, (2\leq j\leq m, j\not=l)\\\lambda_l\in\lambda\sigma_l\\\prod_{i\in\alpha^{-1}(a)\backslash\{1\}}\lambda_i=1\, (1\leq a\leq k)}}
P^{(m)}_{\lambda_m}A_{m-1}P^{(m-1)}_{\lambda_{m-1}}A_{m-2}\ldots P^{(l)}_{\lambda^{-1}\lambda_l}\ldots A_2P^{(2)}_{\lambda_2}(A_1x)\\
&=&\sum_{\substack{\lambda_j\in\sigma_j\, (2\leq j\leq m)\\\lambda_1=\lambda\\\prod_{i\in\alpha^{-1}(a)}\lambda_i=1\, (1\leq a\leq k)}}
P^{(m)}_{\lambda_m}A_{m-1}P^{(m-1)}_{\lambda_{m-1}}A_{m-2}\ldots P^{(l)}_{\lambda_l}\ldots A_2P^{(2)}_{\lambda_2}A_1P^{(1)}_{\lambda_1}x\\
&=&\sum_{\substack{\lambda_j\in\sigma_j\, (1\leq j\leq m)\\\prod_{i\in\alpha^{-1}(a)}\lambda_i=1\, (1\leq a\leq k)}}
P^{(m)}_{\lambda_m}A_{m-1}P^{(m-1)}_{\lambda_{m-1}}A_{m-2}\ldots A_1P^{(1)}_{\lambda_1}x.
\end{eqnarray*}

Since the eigenvectors of $T_1$ corresponding to unimodular
eigenvalues span $X$ and the entangled Ces\`aro means are uniformly
bounded, we obtain the convergence on 
$X$ to the required limit.
\end{proof}

It is interesting that in Theorem $\ref{AP}$ it is not important
whether all operators $T_j$ ($1\leq j\leq m-1$) are different or equal, and the same for
the operators $A_j$, as the following shows. 
We shall need this fact in the proof of Theorem \ref{thm-general-case} below. 

\begin{A}\label{cor:several-op-rel-comp}
Suppose we know that Theorem $\ref{AP}$ and the form of the limit
given by $(\ref{limitEq})$ hold under additional assumption $T_1=\ldots=T_{m-1}$ and $A_1=\ldots=A_{m-1}$. Then it also holds in full generality.
\end{A}

\begin{proof}
Consider the space $\mc{X}:=X^m$ with the diagonal operators
\begin{eqnarray*}
  \mc{T}&:=&\mathrm{diag}(T_1,T_2,\ldots,T_{m-1},I)\in\mc{L}(\mc{X}), \\
  \mc{S}&:=&\mathrm{diag}(I,I,\ldots,I,T_m)\in\mc{L}(\mc{X})
\end{eqnarray*}
and the off-diagonal operator
$$
 \mc{A}:=((\delta_{a-1,b}A_a))_{a,b}\in\mc{L}(\mc{X}),
$$ 
where $\delta$ is the Kronecker symbol. Since all the $T_j$'s ($1\leq
j\leq m-1$) are almost periodic, so is $\mc{T}$, and $\mc{S}$ is clearly totally ergodic on $\mc{X}$. We can hence apply
the weaker statement of Theorem \ref{AP} to the operators $\mc{T}$, $\mc{S}$ and $\mc{A}$. For the vector $(x,0,\ldots,0)^T\in\mc{X}$ for some $x\in X$, this yields the existence of
\begin{equation}\label{eqnPowerSpace}
\lim_{N\to\infty}\frac{1}{N^k}\sum_{n_1,\ldots,n_k=1}^{N}\mc{S}^{n_{\alpha(m)}}\mc{A}\mc{T}^{n_{\alpha(m-1)}}\mc{A}\ldots \mc{A}\mc{T}^{n_{\alpha(1)}}(x,0,\ldots,0)^T.
\end{equation}
However, each of the summands has the form
\[
(0,\ldots,0,T_m^{n_{\alpha(m)}}A_{m-1}T_{m-1}^{n_{\alpha(m-1)}}A_{m-2}\ldots A_1T_1^{n_{\alpha(1)}}x)^T,
\]
and hence the convergence of $(\ref{eqnPowerSpace})$ in the last coordinate implies the required convergence of $(\ref{entEq})$.
Concerning the explicit form of the limit in question, the expression $(\ref{limitEq})$ yields that it is the last coordinate of

\begin{equation}\label{eqnPowerSpace2}
\sum_{\substack{\lambda_j\in\sigma\, (1\leq j\leq m)\\\prod_{i\in\alpha^{-1}(a)}\lambda_i=1\, (1\leq a\leq k)}}
\mc{P}^{\mc{S}}_{\lambda_m}\mc{A}\mc{P}^{\mc{T}}_{\lambda_{m-1}}\ldots \mc{A}\mc{P}^{\mc{T}}_{\lambda_1}(x,0,\ldots,0)^T,
\end{equation}
where $\mc{P}^{\mc{S}}_\lambda$ and $\mc{P}^{\mc{T}}_\lambda$ are the mean ergodic projections onto the eigenspace corresponding to $\lambda$ of $\mc{S}$ and $\mc{T}$, respectively,. Due to the diagonality of $\mc{T}$ and $\mc{S}$, each of the components in $\mc{X}=X^m$ is $\mc{T}$- and $\mc{S}$-invariant, hence we in fact have 
$\mc{P}^{\mc{T}}_\lambda=\mathrm{diag}(P^{(1)}_\lambda,\ldots,P^{(m-1)}_\lambda,\mathds{1}_{\{1\}}(\lambda)I)$ and
$\mc{P}^{\mc{S}}_\lambda=\mathrm{diag}(\mathds{1}_{\{1\}}(\lambda)I,\ldots,\mathds{1}_{\{1\}}(\lambda)I,P^{(m)}_\lambda)$
 where $\mathds{1}_M$ denotes the charasteristic function of a set $M$.
The summands in $(\ref{eqnPowerSpace2})$ thus have the form
\[
(0,\ldots,0,P^{(m)}_{\lambda_m}A_{m-1}P^{(m-1)}_{\lambda_{m-1}}A_{m-2}\ldots A_1P^{(1)}_{\lambda_1}x)^T.
\]
Taking into consideration that the mean ergodic projections satisfy $\mr{Ran}P^{(j)}_{\lambda}=\left\{y\in X|\,T_jy=\lambda y\right\}$, and hence $P^{(j)}_\lambda=0$ whenever
$\lambda\not\in\sigma_j$, the limit reduces to the required form.

\end{proof}

%
%
%
%

\section{The general case}\label{section:general-case}

We now extend the results from the previous section to the
case when only the orbits along the family of operators 
$\left\{A_iT_i^n\right\}_{n\in\mb{N}^+}$ are relatively compact
for all $1\leq i\leq m-1$, i.e., assuming (A2). 

The key for our considerations will be the following extended version
of a classical decomposition theorem, see e.g.~Krengel \cite[Section
2.2.4]{krengel:1985} or \cite[Theorem II.4.8]{eisner-book}. 

\begin{T}\label{J-Gl-dL}\emph{(Jacobs--Glicksberg--de Leeuw decomposition)}
Let $X$ be a Banach space and let $T\in\mathcal{L}(X)$ have relatively weakly compact orbits. Then $X=X_r\oplus X_s$, where
\begin{eqnarray*}
X_r &:=& \overline{\emph{lin}}\left\{x \,\in\, X|\ Tx= \lambda  x \text{ for some } \lambda \in \mb{T} \right\}, \\
X_s &:=& \left\{ x \in X|\, \lim_{j\to\infty}T^{n_j}x=0 \text{ weakly
  for some sequence } \{n_j\}_{j=1}^\infty \text{ with density }1 \right\},
\end{eqnarray*}
with both subspaces being invariant under $T$.
In addition, if $X'$ is separable, then there exists a sequence $\{n_j\}_{j=1}^\infty$ with
density $1$ such that $\lim_{j\to\infty}T^{n_j}|_{X_s}=0$ weakly.
\end{T}
\noindent Recall that the density of a set $M\subset \mb{N}$ is
defined by 
$$
  d(M)=\lim_{n\to\infty} \frac{|M\cap\{1,\ldots,n\}|}{n}\leq 1,
$$
whenever the above limit exists. 

We further need the following well-known lemmas. 

\begin{Le}\label{lemma:K-vN}\emph{(Koopman--von Neumann)}
For a bounded sequence $\{a_n\}_{n=1}^\infty\subset[0,\infty)$ the following assertions are equivalent.
\begin{enumerate}[(a)]
\item $\displaystyle\lim_{n\to\infty}\frac{1}{n}\sum_{k=1}^n a_k=0$.
\item There exists a subsequence $\{n_j\}_{j=1}^\infty$ of $\mb{N}$ with density $1$ such that $\lim_{j\to\infty}a_{n_j}=0$.
\end{enumerate}
\end{Le}
\noindent We refer to e.g.~Petersen \cite[p.~65]{petersen:1983} for
the proof. 

\begin{Le}\label{lemma:rel-comp}
Let $X$ be a Banach space and let $\{T_n\}_{n=1}^\infty,\{S_n\}_{n=1}^\infty
\subset \mathcal{L}(X)$. Then the following assertion holds. 
If both $\{T_n x:\,n\in\mb{N}\}$ and  $\{S_n x:\,n\in\mb{N}\}$
  are relatively compact in $X$ for every $x\in X$, then so is
  $\{T_nS_n x:\,n\in\mb{N}\}$ for every $x\in X$.
\end{Le}
\begin{proof}
 Since compact sets are bounded and by the uniform boundedness principle, there exists $M>0$ such that
$\|T_n\|\leq M$ and $\|S_n\|\leq M$ holds for every $n\in\mb{N}^+$. Take now $x\in X$ and a sequence $\{n_j\}\subset \mb{N}$. Then there exists a
subsequence $\{m_j\}$ of $\{n_j\}$ such that $\lim_{j\to\infty} S_{m_j}x=y$ for
some $y\in X$. Furthermore, there exists a subsequence of  $\{m_j\}$
which we again denote by  $\{m_j\}$ such that $\lim_{j\to\infty}
T_{m_j}y=z$ for some $z\in X$. This yields  
$$
\|T_{m_j}S_{m_j}x-z\|\leq M\|S_{m_j}x-y\|+\|T_{m_j}y-z\|\to 0 \quad
\text{as } j\to \infty,
$$
proving relative compactness of $\{T_nS_n x:\,n\in\mb{N}\}$.
\end{proof}

The following is the main result of the paper.

\begin{T}\label{thm-general-case}
Let $X$ be a Banach space and $\alpha:\{1,\ldots, m\}\to\{1,\ldots
k\}$ be surjective for some $k\leq m$. Let further $T_1,\ldots, T_m, A_1,\ldots,
A_{m-1}\in \mathcal{L}(X)$ satisfy assumptions (A1) and (A2).  
Then the entangled ergodic averages 
\begin{equation*}
\frac{1}{N^k}\sum_{n_1,\ldots,n_k=1}^{N}T_m^{n_{\alpha(m)}}A_{m-1}T_{m-1}^{n_{\alpha(m-1)}}A_{m-2}\ldots A_{1}T_1^{n_{\alpha(1)}}
\end{equation*}
converge strongly, and their limit is given by
$$
\sum_{\substack{\lambda_j\in\sigma_j\, (1\leq j\leq m)\\\prod_{i\in\alpha^{-1}(a)}\lambda_i=1\, (1\leq a\leq k)}}
P^{(m)}_{\lambda_m}A_{m-1}P^{(m-1)}_{\lambda_{m-1}}A_{m-2}\ldots A_1P^{(1)}_{\lambda_1},
$$ 
where $\sigma_j=P_\sigma(T_j)\cap \mb{T}$ and $P^{(j)}_{\lambda_j}$ is
the projection onto the eigenspace of $T_j$ corresponding to $\lambda_j$, i.e., the mean
ergodic projection of the operator $\overline{\lambda_j}T_j$.
\end{T}

\begin{proof}
As in the proof of Corollary \ref{cor:several-op-rel-comp} we may assume $T_j=T$ and $A_j=A$
for some $T,A\in\mathcal{L}(X)$ and all $1\leq j\leq m-1$. For $x\in
X$, we have to show convergence of
\begin{equation}\label{eq:T-A}
\frac{1}{N^k}\sum_{n_1,\ldots,n_k=1}^{N}T_m^{n_{\alpha(m)}}AT^{n_{\alpha(m-1)}}A\ldots AT^{n_{\alpha(1)}}x.
\end{equation}
By Theorem \ref{J-Gl-dL}, the summands in (\ref{eq:T-A}) satisfy
\begin{eqnarray*}
T_m^{n_{\alpha(m)}}AT^{n_{\alpha(m-1)}}A&\ldots& AT^{n_{\alpha(1)}}x \\
&=&
\sum_{a=1}^{m-1} T_m^{n_{\alpha(m)}}A\ldots AT^{n_{\alpha(a)}}P_s A
T^{n_{\alpha(a-1)}}P_rA \ldots AT^{n_{\alpha(1)}}P_rx\\ 
&+&
T_m^{n_{\alpha(m)}}AT^{n_{\alpha(m-1)}}P_rA\ldots AT^{n_{\alpha(1)}}P_rx,
\end{eqnarray*}
where $P_r$ and $P_s$ are the projections onto $X_r$ and $X_s$ pertaining to $T$ from
Theorem \ref{J-Gl-dL}, respectively. By Theorem \ref{AP}, the averages
of the second summand above converge to the desired limit. It remains
to show that the averages of the
first summand converge to $0$, i.e., that for every $x\in X$ and
$1\leq a\leq m-1$ one has
\begin{equation}\label{eq:zero}
\lim_{N\to\infty} \frac{1}{N^k}\sum_{n_1,\ldots,n_k=1}^{N}
T_m^{n_{\alpha(m)}}AT^{n_{\alpha(m-1)}}A\ldots AT^{n_{\alpha(a)}}P_s A T^{n_{\alpha(a-1)}}P_rA \ldots
AT^{n_{\alpha(1)}}P_rx=0. 
\end{equation}

Consider
$$
  K:=\{A T^{n_{a-1}}P_rA \ldots AT^{n_1}P_rx|\ n_{a-1},\ldots,n_1\in\mb{N}\}
$$
which is relatively compact by assumption and Lemma
\ref{lemma:rel-comp}(a). We now show that the dual space of the smallest $T$-invariant subspace $Y$
containing $K$ is separable. Observe first that
$Y=\clin\{T^nx|\,n\in\mb{N},\, x\in K\}$. We first show that the set 
$\mathrm{Orb}(K):=\{T^nx|\,n\in\mb{N},\, x\in K\}$ is relatively weakly compact. Take a sequence $\{T^{n_j}x_j\}_{j=1}^\infty$ with $x_j\in K$ and $n_j\in\mb{N}$. Since $K$ is relatively compact, there exists a subsequence of $\{x_j\}$ (which we again denote by $\{x_j\}$) converging to some $z$. Moreover, since $T$ has relatively weakly compact orbits, there is a subsequence of $\{n_j\}$ (which we again denote by $\{n_j\}$) such that $\lim_{j\to\infty}T^{n_j}z=w$ weakly for some $w\in X$. So we have
$$
|\langle T^{n_j}x_j-w,y \rangle|\leq |\langle T^{n_j}x_j - T^{n_j}z,y\rangle| + |\langle T^{n_j}z-w,y\rangle| \to 0\quad \forall y\in X',
$$
i.e.,  $\lim_{j\to\infty}T^{n_j}x_j=w$ weakly, and therefore 
$\mathrm{Orb}(K)$ is relatively weakly
compact. Since $Y$ is separable, and the weak topology is metrisable on
weakly compact subsets of separable spaces (see e.g.~Dunford, Schwartz
\cite[Theorem V.6.3]{dunford/schwartz:1958}), the weak topology on $\mathrm{Orb}(K)$ is metrisable. So it is induced by countably many $\{y_n\}\subset Y'$, implying separability of $Y'$ by $Y=\clin\mathrm{Orb}(K)$.

Theorem \ref{J-Gl-dL} assures now the existence of a sequence
$\{n_j\}\subset \mb{N}$ with density $1$ such that $\lim_{j\to\infty} T^{n_j}P_sy=0$ weakly for every  $y\in K$.
This implies
\begin{equation}\label{eq:A-weak-conv}
\lim_{j\to\infty} AT^{n_j}P_sy=0 \text{ weakly} \quad \text{for every } y\in
K.
\end{equation}
We now show that 
\begin{equation}\label{eq:A-strong-conv}
   \lim_{j\to\infty} \|A T^{n_j}P_sy\|=0 \quad \text{uniformly in } y\in K.
\end{equation}
Indeed, assume that for some $y\in K$, $AT^{n_j}P_s y$ does not
converge strongly to zero. Then there exist $\delta>0$ and a
subsequence $\{m_j\}$ of $\{n_j\}$ such that $\|AT^{m_j}P_s y\|\geq
\delta$ for every $j$. By relative compactness of $\{A T^{n}P_sy:\,
n\in \mb{N} \}$, the sequence $\{AT^{m_j}P_s y\}_{j=1}^\infty$ has a
strong accumulation point which by (\ref{eq:A-weak-conv}) must be zero, a
contradiction. Thus, $\lim_{j\to\infty} \|A T^{n_j}P_sy\|=0$ for every
$y\in K$ and therefore uniform in $y\in K$, since strong convergence in $\mc{L}(X)$ implies uniform strong convergence on compact sets. 
Since the sequence
$\{n_j\}$ has density $1$, the equation (\ref{eq:zero}) follows from
(\ref{eq:A-strong-conv}) and 
Lemma \ref{lemma:K-vN}.
\end{proof}

%
%
%
%

\section{Connection to convergence of multiple ergodic averages}

In this section we discuss connection between the considered above
entangled ergodic theorems and the important topic as
multiple ergodic averages. The latter is concerned with
non-commutative dynamical systems and was studied by Niculescu, Str\"oh and Zsid\'o \cite{niculescu/stroh/zsido}, Duvenhage \cite{duvenhage},
Beyers, Duvenhage and Str\"oh \cite{beyers/duvenhage/stroh}, Fidaleo \cite{fidaleo:2009}, and Austin, Eisner, Tao \cite{TTT}. 

We first introduce what we mean by a non-commutative dynamical system
and corresponding multiple ergodic averages. 

\begin{D}
A von Neumann (or non-commutative) dynamical system is a triple
$(\M, \tau, \a)$, where $\M$ is a von Neumann algebra,
$\tau:\M\to \mb{C}$ is a faithful normal trace, and    
$\a:\M\to \M$ is a $\tau$-preserving $*$-automorphism.
We say for $k\in\mb{N}$ that the multiple ergodic averages 
\begin{equation}\label{eq:mult-erg-ave}
\frac{1}{N}\sum_{n=1}^N \a^n(a_1)\a^{2n}(a_2) \ldots \a^{kn}(a_k)
\end{equation}
converge \emph{strongly} if they converge in the 
$\tau$-norm
defined by $\|a\|_\tau:=\sqrt{\tau(aa^*)}$. The averages in
(\ref{eq:mult-erg-ave}) are called \emph{weakly} convergent if 
$$
  \frac{1}{N}\sum_{n=1}^N \tau(a_0\a^n(a_1)\a^{2n}(a_2) \cdots \a^{kn}(a_k))
$$
converges as $N\to\infty$ for every $a_0\in \M$. 
\end{D}

We recall that by the Gel'fand--Neumark--Segal theory, $\M$ can be
identified with a dense subspace of a Hilbert space, where the Hilbert
space can be obtained as the completion of $\M$ with respect to the
$\tau$-norm. Thus, identifying elements of $\M$ with elements in $H$
and by the standard density argument, strong convergence of the multiple ergodic averages
(\ref{eq:mult-erg-ave}) corresponds to norm convergence in $H$ and
weak convergence of (\ref{eq:mult-erg-ave}) corresponds
to weak convergence in $H$. 

Recall further that for the automorphism
$\a$ there exists a unitary operator $u\in\mc{L}(H)$ such that
$\a(a)=uau^{-1}$, see e.g. \cite[Prop. 4.5.3]{kadison/ringrose}. (Note that $u$ does not
necessarily belong to $\M$.) Thus, averages (\ref{eq:mult-erg-ave}) take the form 
\begin{equation}\label{eq:ent-ave-for-mult}
\frac{1}{N}\sum_{n=1}^N u^na_1u^na_2\cdots u^n a_k u^{-kn},
\end{equation}
i.e., are a special case of entangled ergodic averages for the
constant partition $\alpha(j)=1$ for every $j\in{1,\ldots,k}$ and the
operators $u,\ldots, u, u^{-k}$. 

It is well-known that strong (weak) topology and strong (weak)
operator topology on $\M$ coincide. 
Therefore, there is a direct correspondence between strong (weak) convergence of
multiple ergodic averages (\ref{eq:mult-erg-ave}) and strong (weak)
\emph{operator} convergence of the entangled ergodic averages
(\ref{eq:ent-ave-for-mult}), cf. also Fidaleo \cite{fidaleo:2009}. 

\begin{A}
Let $(\M, \tau, \a)$ be a von Neumann dynamical system and $H$ and $u$ as
above. Let further $a_1,\ldots,a_k\in \M$. Then the multiple ergodic averages
(\ref{eq:mult-erg-ave}) converge strongly (weakly) if and only if the
entangled averages (\ref{eq:ent-ave-for-mult}) converge  
in the strong (weak) operator topology. 
\end{A}

As was shown in \cite{TTT}, multiple ergodic averages
(\ref{eq:mult-erg-ave}) do not converge 
in general for $k\geq 3$. Theorem \ref{thm-general-case} shows now that for every von Neumann
dynamical system there is a class $\mc{K}$ depending on the system
such that the multiple ergodic averages converge strongly whenever
$a_1,\ldots,a_k\in \mc{K}$, and $\mc{K}$ can be chosen as the 
subspace of all elements $a\in \M$ such that $\{a u^n: n\in \mb{N}\}$ is relatively compact in
  $\mc{L}(H)$ for the strong operator topology.

\begin{R}\rm
Note that also more general sequences of powers than arithmetic sequences for the multiple ergodic averages can be treated by a slight modification in the last inductive step of the proof of Theorem \ref{AP}.
For instance, consider the averages
 \begin{equation}\label{eq:ent-ave-for-mult2}
\frac{1}{N^k}\sum_{n_1,\ldots,n_k=1}^{N} \beta^{n_{\alpha(1)}}(a_1)\beta^{n_{\alpha(1)+\alpha(2)}}(a_2)\cdots \beta^{\sum_{j=1}^m n_{\alpha(j)}}(a_m).
\end{equation}
These can then be rewritten as
\begin{equation}
 \frac{1}{N^k}\sum_{n_1,\ldots,n_k=1}^{N} u^{n_{\alpha(1)}}a_1 u^{\alpha(2)}a_2\cdots u^{n_{\alpha(m)}}a_m u^{-\sum_{j=1}^m n_{\alpha(j)}}.
\end{equation}
The last exponent being a sum of $\alpha(j)$'s rather than a single one does not matter on the almost weakly stable part, the compactness arguments of the proof of Theorem \ref{thm-general-case} still hold. On the almost periodic part however, suppose $x$ is an eigenvector to the unimodular eigenvalue $\lambda$. Then
\begin{eqnarray*}
&& \frac{1}{N^k}\sum_{n_1,\ldots,n_k=1}^{N} u^{n_{\alpha(1)}}a_1 u^{\alpha(2)}a_2\cdots u^{n_{\alpha(m)}}a_m u^{-\sum_{j=1}^m n_{\alpha(j)}}x\\&=&\frac{1}{N^k}\sum_{n_1,\ldots,n_k=1}^{N} u^{n_{\alpha(1)}}a_1 u^{\alpha(2)}a_2\cdots u^{n_{\alpha(m)}}a_m \lambda^{-\sum_{j=1}^m n_{\alpha(j)}}x,
\end{eqnarray*}
and the powers of the eigenvalue $\lambda$ have to be pulled forward to not only a single operator as in the inductive proof of Theorem \ref{AP}, but distributed amongst all of them to get back to the standard form:
\begin{eqnarray*}
&& \frac{1}{N^k}\sum_{n_1,\ldots,n_k=1}^{N} u^{n_{\alpha(1)}}a_1 u^{\alpha(2)}a_2\cdots u^{n_{\alpha(m)}}a_m \lambda^{-\sum_{j=1}^m n_{\alpha(j)}}x\\&=&
\frac{1}{N^k}\sum_{n_1,\ldots,n_k=1}^{N} (\lambda^{-1}u)^{n_{\alpha(1)}}a_1(\lambda^{-1}u)^{\alpha(2)}a_2\cdots (\lambda^{-1}u)^{n_{\alpha(m)}}a_m x.
\end{eqnarray*}
Hence the averages (\ref{eq:ent-ave-for-mult2}) converge strongly if $\{a_j u^n: n\in \mb{N}\}$ is relatively compact in
  $\mc{L}(H)$ for the strong operator topology for every $1\leq j\leq m$.
\end{R}

%
%
%
%

\section{Continuous case}\label{section:continuous-case}

In this section we treat the continuous time scale, 
where the operators $T_j$ and their powers are replaced by strongly
continuous ($C_0$-) semigroups $(T_j(t))_{t\geq 0}$. 
The study of the
continuous version for entangled ergodic averages seems to be
new. Some steps in the proofs are similar to the discrete case and will be skipped. 
For the general theory of strongly continuous semigroups we refer to e.g. Engel, Nagel
\cite{engel/nagel:2000}. 
For a semigroup $(T(t))_{t\geq 0}$ we often write $T(\cdot)$.

\begin{D}\label{def:alm-per-cont}
A $C_0$-semigroup of operators $(T(t))_{t\geq0}\subset\mc{L}(X)$
acting on a Banach space $X$ is called \emph{almost periodic} if it is bounded (i.e.~$\sup_{t\geq0}\|T(t)\|<\infty$) and satisifies 
\[
X=\overline{\mathrm{lin}}\left\{x\in X\big|\exists
  \varphi\in\mb{R}: T(t)x=e^{i\varphi t} x\,\forall\,t\geq0\right\}.
\]
\end{D}

\noindent Recall that by the spectral mapping theorem (see e.g.~\cite[Corollary IV.3.8]{engel/nagel:2000}), $T(t)x=e^{i\varphi
  t}x$ for every $t\geq 0$ if and only if $Bx=i\varphi x$ for the
generator $B$ of $T(\cdot)$.

\begin{T}\label{AP-cont}
Let $X$ be a Banach space, $T_1(\cdot),\ldots,T_{m-1}(\cdot)$ be
almost periodic $C_0$-semigroups on $X$, $T_m(\cdot)$ a 
bounded totally ergodic $C_0$-semigroup on $X$
and $A_1,\ldots,A_{m-1}\in\mc{L}(X)$. 
Then the entangled Ces\`aro means
\begin{equation}\label{entEq-cont}
\frac{1}{t^k}\int_{[0,t]^k}T_m(s_{\alpha(m)})A_{m-1}T_{m-1}(s_{\alpha(m-1)})A_{m-2}\ldots A_1T_1(s_{\alpha(1)})\,ds_1\,\ldots\,ds_k
\end{equation}
converge strongly as $t\to\infty$.
\end{T}
\noindent The integrals in (\ref{entEq-cont}) are defined strongly. Recall that a semigroup $(T(t))_{t\geq 0}$ is called totally ergodic if the semigroup $(e^{i\varphi t} T(t))_{t\geq 0}$ is mean ergodic for
every $\varphi\in\mb{R}$.

\begin{proof}
Since almost periodic $C_0$-semigroups are totally ergodic by the standard
density argument, 
we can for each $T_j(\cdot)$ ($1\leq j\leq m$) define the
projections $P^{(j)}_\varphi$ as the mean ergodic projections of
$(e^{-i\varphi t}T_j(t))_{t\geq0}$ with range
$\{y\in X|\,T_j(t)y=e^{i\varphi t} y\,\forall\,t\geq0\}$. 
Let $\sigma_j$ denote the point spectrum of the generator $B_j$ of the
semigroup $T_j(\cdot)$ on $i\mb{R}$ ($1\leq j\leq m$).

The proof can then be concluded by an induction argument on $m$ 
analogous to the discrete case
 showing that the limit of the entangled Ces\`aro means $(\ref{entEq-cont})$ is given by
\begin{equation}
 \sum_{\substack{\varphi_j\in\sigma_j\, (1\leq j\leq m)\\\sum_{i\in\alpha^{-1}(a)}\varphi_i=0\, (1\leq a\leq k)}}
P^{(m)}_{\varphi_m}A_{m-1}P^{(m-1)}_{\varphi_{m-1}}A_{m-2}\ldots A_1P^{(1)}_{\varphi_1}.
\end{equation}
\end{proof}

We shall use the following analogue of Theorem \ref{J-Gl-dL}, see e.g. \cite[Theorem III.5.7]{eisner-book}.

\begin{T}\label{J-Gl-dL-cont}\emph{(Continuous Jacobs--Glicksberg--de Leeuw decomposition)}
Let $X$ be a Banach space and let $T(\cdot)\subset\mathcal{L}(X)$ be a
$C_0$-semigroup with relatively weakly compact orbits, 
i.e., such that $\{T(t)x:\, t\in [0,\infty)\}$ is relatively compact in $X$ in the weak
topology for every $x\in X$. 
Then $X=X_r\oplus X_s$, where
\begin{eqnarray*}
X_r &=& \overline{\emph{lin}}\left\{x \,\in\, X|\ T(t)x= e^{i\varphi t}  x\,\forall\,t\geq0 \text{ for some } \varphi \in \mb{R} \right\}, \\
X_s &=& \left\{ x \in X|\, \lim_{M\ni j\to\infty}T(t)x=0 \text{ weakly
  for some } M\subset[0,\infty) \text{ with density }1 \right\},
\end{eqnarray*}
with both subspaces being invariant under $T$.
In addition, if $X'$ is separable, then there exists a set $M\subset[0,\infty)$ with
density $1$ such that $\lim_{M\ni t\to\infty}T(t)|_{X_s}=0$ weakly.
\end{T}
\noindent The density of a set $M\subset [0,\infty)$ is
defined by 
\[
  d(M)=\lim_{t\to\infty} \frac{\lambda(M\cap[0,t])}{t}\leq 1,
\]
with $\lambda(\cdot)$ denoting the Lebesgue measure, whenever the above limit exists. 

We further need the following continuous version of Lemma
\ref{lemma:K-vN}.

\begin{Le}\label{lemma:K-vN-cont}\emph{(Koopman--von Neumann,
    continuous version)}
For a continuous function $f:[0,\infty)\to[0,\infty)$ the following assertions are equivalent.
\begin{enumerate}[(a)]
\item $\displaystyle\lim_{t\to\infty}\frac{1}{t}\int_{[0,t]} f(s)\,ds=0$.
\item There exists a subset $M$ of $[0,\infty)$ with density $1$ such that $\lim_{s\in M,\,s\to\infty} f(s)=0$.
\end{enumerate}
\end{Le}
\noindent For the proof, which is analogous to the discrete case, see
e.g.~\cite[Lemma III.5.2]{eisner-book}.

\vspace{0.1cm}

The main result of this section is the following. 

\begin{T}\label{thm-general-case-cont}
Let $X$ be a Banach space, $m\in\mb{N}$, 
$T_m(\cdot),\ldots, T_{m}(\cdot)$ be $C_0$-semigroups on $X$, 
$A_1,\ldots,
A_{m-1}\in \mathcal{L}(X)$, and $\alpha:\{1,\ldots, m\}\to\{1,\ldots ,
k\}$ be a surjective map for some $k,m\in\mb{N}$.  Assume the following. 
\begin{itemize}
\item[(A1$_c$)] The semigroup $T_m(\cdot)$ is bounded and totally ergodic and $T_j(\cdot)$ has 
relatively weakly compact orbits for every  
$1\leq j\leq m-1$.
\item[(A2$_c$)] Every $A_j$ is compact on the orbits of $T_j(\cdot)$, i.e., 
$\{A_jT_j(t) x:\, t\in [0,\infty)\}$ is relatively compact in $X$
for every $x\in X$ and $1\leq j\leq m-1$.
\end{itemize}
Then
the entangled ergodic averages 
\begin{equation*}
\frac{1}{t^k}\int_{[0,t]^k}T_m(s_{\alpha(m)})A_{m-1}T_{m-1}(s_{\alpha(m-1)})A_{m-2}\ldots A_{1}T_1(s_{\alpha(1)})\,ds_1\,\ldots\,ds_k
\end{equation*}
converge strongly. Denoting the generator of $T_j(\cdot)$ by $B_j$ ($1\leq j\leq m$), the strong limit is given by the formula
\[
\sum_{\substack{\varphi_j\in\sigma_j\, (1\leq j\leq m)\\\sum_{i\in\alpha^{-1}(a)}\varphi_i=0\, (1\leq a\leq k)}}
P^{(m)}_{\varphi_m}A_{m-1}P^{(m-1)}_{\varphi_{m-1}}A_{m-2}\ldots A_1P^{(1)}_{\varphi_1},
\] 
where $\sigma_j=P_\sigma(B_j)\cap i\mb{R}$ and $P^{(j)}_{\varphi_j}$ is the projection onto the eigenspace
of $B_j$ corresponding to $i\varphi_j$, i.e., the mean
ergodic projection of the semigroup $(e^{-i\varphi_j t}T_j(t))_{t\geq0}$.
\end{T}

\begin{proof}
Using the arguments from Proposition \ref{cor:several-op-rel-comp}, we may again assume that we have $T_j(\cdot)=T(\cdot)$ and $A_j=A$ for $1\leq j\leq m-1$. It is to be shown that

\begin{equation}\label{eq:etwas}
 \frac{1}{t^k}\int_{[0,t]^k}T_m(s_{\alpha(m)})AT(s_{\alpha(m-1)})A\ldots AT(s_{\alpha(1)})x\,ds_1\,\ldots\,ds_k
\end{equation}
converges for every $x\in X$. By Theorem \ref{J-Gl-dL-cont}, the integrand can be split with the help of the projections $P_r$ and $P_s$ onto $X_r$ and $X_s$, respectively, and we have

\begin{eqnarray*}
&&T_m(s_{\alpha(m)})AT(s_{\alpha(m-1)})A\ldots AT(s_{\alpha(1)})x
\\&=& \sum_{a=1}^{m-1}T_m(s_{\alpha(m)})A\ldots AT(s_{\alpha(a)})P_sAT(s_{\alpha(a-1)})P_rA\ldots AT(s_{\alpha(1)})P_rx
\\&+& T_m(s_{\alpha(m)})AT(s_{\alpha(m-1)})P_rA\ldots AT(s_{\alpha(1)})P_rx.
\end{eqnarray*}

The integral means of the second term converge by Theorem \ref{AP-cont} to the desired limit, and so it is enough to show that the rest converges in mean to $0$, i.e. for every $x\in X$ and
$1\leq a\leq m-1$ one has
\begin{equation}\label{eq:zero-cont}
\lim_{t\to\infty} \frac{1}{t^k}\int_{[0,t]^k}
T_m(s_{\alpha(m)})A\ldots AT(s_{\alpha(a)})P_s A T(s_{\alpha(a-1)})P_rA \ldots
AT(s_{\alpha(1)})P_rx=0. 
\end{equation}

\noindent Consider
$$
  K:=\{A T(s_{a-1})P_rA \ldots AT(s_1)P_rx|\ s_{a-1},\ldots,s_1\in[0,\infty)\}.
$$
This set is relatively compact by Lemma \ref{lemma:rel-comp} and the
assumption. As in the discrete case, one can show that the dual space
of the smallest $T(\cdot)$-invariant subspace $Y$ containing $K$ is
separable. Note that the separability of $Y$ itself follows from the strong continuity of the semigroup, as it yields a dense countable subset 
\[
\{A T(s_{a-1})P_rA \ldots AT(s_1)P_rx|\ s_{a-1},\ldots,s_1\in[0,\infty)\cap\mb{Q}\} 
\]
of $K$.

Theorem \ref{J-Gl-dL-cont} then assures the existence of a set $M\subset [0,\infty)$ with density $1$ such
that
$$
  \lim_{s_j\in M,\,s_j\to\infty} T(s_j)P_sy=0\quad  \text{weakly} \quad \text{for
    every } y\in K
$$
implying $\lim_{s_j\in M,\,s_j\to\infty} AT(s_j)P_sy=0$ weakly for every $y\in K$.
Since $\{A T(s)P_sy:\, s\in [0,\infty) \}$ is relatively compact, and strong convergence in $\mc{L}(X)$ implies uniform strong convergence on compact subsets of $X$, we obtain
$$\lim_{s_j\in M,\, s_j\to\infty} \|A T(s_j)P_sy\|=0 \quad \text{uniformly in } y\in K.$$
Since the set
$M$ has density $1$, the equation (\ref{eq:zero-cont}) follows from
Lemma \ref{lemma:K-vN-cont}.
\end{proof}

Note that, as in the discrete case, the class of semigroups satisfying
assumption (A1$_c$) is large including e.g.~bounded $C_0$-semigroups on reflexive Banach spaces.

\vspace{0.4cm}

\noindent\textbf{Acknowledgement.} The authors are grateful to Rainer Nagel and Marco Schreiber for valuable discussions and comments.


\end{document}